\documentclass{amsart}

\newcommand{\Q}{\mathbb Q}

\newtheorem{thm}{Theorem}

\title{Rational Quartic Reciprocity}
\author{Franz Lemmermeyer}
\address{M\"orikeweg 1, 73489 Jagstzell, Germany}
\email{hb3@ix.urz.uni-heidelberg.de}

\begin{document}

\begin{abstract}
 We provide a simple proof of the general rational quartic reciprocity
 law due to Williams, Hardy and Friesen.
\end{abstract}

\maketitle

In 1985, K. S. Williams, K. Hardy and C. Friesen \cite{11} published 
a reciprocity formula that comprised all known rational quartic 
reciprocity laws. Their proof consisted in a long and complicated 
manipulation of Jacobi symbols and was subsequently simplified (and
generalized) by R. Evans \cite{3}. In this note we will give a proof
of their reciprocity law which is not only considerably shorter but
which also sheds some light on the raison d'\^etre of rational
quartic reciprocity laws. For a survey on rational reciprocity laws,
see E. Lehmer \cite{7}.

We want to prove the following

\begin{thm}
Let $m \equiv 1 \bmod 4$ be a prime, and let $A, B, C$ be integers
such that
$$ \begin{array}{rl} A^2=m(B^2+C^2),    \qquad & 2 \mid B, \\
   (A,B) = (B,C) = (C,A) = 1,\qquad & A+B \equiv 1 \bmod 4.
\end{array} $$
Then, for every odd prime $p>0$ such that $(m/p) = + 1$, 
\begin{equation}\label{E1}
\left(\frac{A+B\sqrt m}{p}\right) =  \left(\frac{p}{m} \right)_4 . 
\end{equation}
\end{thm}

\begin{proof}
Let $k = \Q(\sqrt{m}\,)$; then $K = \Q(\sqrt{m}, \sqrt{A+B\sqrt{m}}\,)$
is a quartic cyclic extension of $\Q$ containing $k$, as can be verified 
quickly by noting that 
$$ A^2-mB^2 = mC^2 = (\sqrt{m}C)^2 \quad \text{ and } 
   \quad \sqrt{m}C \in k \setminus \Q.  $$
We claim that $K$ is the quartic subfield of $\Q(\zeta_m)$, the field
of $m$th roots of unity. This will follow from the theorem of Kronecker 
and Weber once we have seen that no prime $\ne m$ is ramified in $K/\Q$. 
But the identity
\begin{equation}\label{E2}
 2(A+B\sqrt{m})(A+C\sqrt{m}) = (A+B\sqrt{m}+C\sqrt{m})^2 \end{equation}
shows that $K = k\left(\sqrt{2(A+C\sqrt{m})}\right)$, and so the only
odd primes that are possibly ramified in  $K/k$ are common divisors of
$A^2-mB^2 = mC^2$ and $A^2-mC^2 = mB^2$.  Since $B$ and $C$ are assumed
to be prime to each other, only $2$ and $m$ can ramify. Now 
$\sqrt{m} \equiv 1 \bmod 2$ (since $m \equiv 1 \bmod 4$) implies
$B\sqrt{m} \equiv B \bmod 4$, and we see $A+B\sqrt{m} \equiv A+B 
\equiv 1 \bmod 4$, which shows that $2$ is unramified in $K/k$ (and
therefore also in $K/\Q$).

The reciprocity formula will follow by comparing the decomposition laws
in $K/\Q$ and $\Q(\zeta_m)/\Q$: if $(m/p) = +1$, then $p$ splits in
$k/\Q$; if $f > 0$ is the smallest natural number such that
$p^f \equiv 1 \bmod m$ (here we have to assume that $p>0$), then
$p$ splits into exactly $g = (m-1)/f$ prime ideals in $\Q(\zeta_m)$,
and we see
$$ \begin{array}{rl} \left(\frac{p}{m} \right)_4 = 1 
    & \iff  p^{(m-1)/4} \equiv 1 \bmod m \\
    & \iff  f \text{ divides } \frac14(m-1) = \frac14fg \\
    & \iff  g \equiv 0 \bmod 4 \\
    & \iff  \text{the degree of the decomposition field } Z \text{ of } p 
            \text{ is divisible by 4} \\
    & \iff Z \text{ contains } K \text { (because }
             Gal(\Q(\zeta_m)/\Q) \text { is cyclic)} \\
    & \iff  p \text{ splits completely in } K/\Q \\
    & \iff  p \text{ splits completely in } K/k \text{ (since } p
              \text{ splits completely in } k/\Q \text{)} \\
    & \iff \left(\frac{A+B\sqrt m}{p}\right) = 1. 
\end{array} $$
 
  This completes the proof of the theorem. 
\end{proof}

Letting $m=2$ and replacing the quartic subfield of $\Q(\zeta_m)$
used above by the cyclic extension $\Q\left(\sqrt{2+\sqrt{2}}\,\right)$
contained in $\Q(\zeta_{16})$ yields the equivalence
\begin{equation}\label{E3}
 \left(\frac{A+B\sqrt 2}{p}\right) = 1 \iff p \text{ splits in }
   \Q\left(\sqrt{2+\sqrt{2}}\,\right) \iff p \equiv \pm 1 \bmod 16, 
\end{equation}
stated in a slightly different way in \cite{11}.

Formula (\ref{E1}) differs from the one given in \cite{11}, which reads 
\begin{equation}\label{E4} 
\left(\frac{A+B\sqrt m}{p}\right) = 
   (-1)^{(p-1)(m-1)/8}\left(\frac{2}{p}\right)
                  \left(\frac{p}{m}\right)_4, \end{equation}
where $A, B, C > 0,\, B$ is odd, and $C$ is even. Formula (2) shows that
$$ \left(\frac{A+B\sqrt m}{p}\right) = \left(\frac{2}{p}\right)
   \left(\frac{A+C\sqrt m}{p}\right), $$
and so, for $B$ even and $C$ odd, (4) is equivalent to
 \begin{equation}\label{E5}
\left(\frac{A+B\sqrt m}{p}\right) = (-1)^{(p-1)(m-1)/8}
        \left(\frac{p}{m}\right)_4. \end{equation}

Now $A \equiv 1 \bmod 4$ since $A^2 = m(B^2+C^2)$ is the product of
$m \equiv 1 \bmod 4$ and of a sum of two relatively prime squares,
and we have $A+B \equiv 1 \bmod 4 \iff {4 \mid B} \iff m \equiv 1 \bmod 8.$
The sign of $B$ is irrelevant, therefore
$$\left(\frac{-1}{p}\right)^{B/2} = (-1)^{(p-1)(m-1)/8}.$$
This finally shows that (\ref{E1}) is in fact equivalent to (\ref{E4}).

Another version of (\ref{E1}) which follows directly from (\ref{E5}) is
\begin{equation}\label{E6}
\left(\frac{A+B\sqrt m}{p}\right) = 
  \left(\frac{p^*}{m}\right)_4, \end{equation}
where $A, B > 0$ and $p^* = (-1)^{p-1)/2}p$.

Formula (\ref{E1}) can be extended to composite values of $m$ (where the
prime factors of $m$ satisfy certain conditions given in \cite{11})
in very much the same way as Jacobi extended the quadratic reciprocity
law of Gauss; this extension, however, is not needed for deriving the
known rational reciprocity laws of K. Burde \cite{1}, E. Lehmer 
\cite{6,7} and A. Scholz \cite{9}. These follow from (\ref{E1}) by assigning
special values to $A$ and $B$, in other words: they all stem from the
observation that the quartic subfield $K$ of  $\Q(\zeta_m)$ can be
generated by different square roots over $k=\Q(\sqrt{m}\,)$.

The fact that (\ref{E1}) is valid for primes $p \mid ABC$ (which 
has not been proved in \cite{11}) shows that we no longer have to 
exclude the primes $q \mid ab$ in Lehmer's criterion (as was 
necessary in \cite{11}), and it allows us to derive Burde's 
reciprocity law in a more direct way: let $p$ and $q$ be primes 
$\equiv 1 \bmod 4$ such that $p = a^2+b^2$, $q=c^2+d^2$,
$2\mid b$, $2\mid d$, $(p/q) = +1$, and define
$$ A = pq, \quad B = b(c^2-d^2)+2acd, \quad C = a(c^2-d^2)-2bcd, \quad
   m = q.$$
Then $2 \mid B,\, B \equiv 2d(ac+bd) \bmod q$ (since $c^2 \equiv 
-d^2 \bmod q$), the sign of $A$ does not matter (since $q \equiv 
1 \bmod 4$), and so formula (\ref{E1}) yields
$$ \left(\frac{q}{p}\right)_4 = \left(\frac{A+B\sqrt p}{q}\right) 
   = \left(\frac{B}{q}\right)\left(\frac{p}{q}\right)_4. $$
Now the well known $(\frac{2d}{q}) = +1$ implies Burde's law
\begin{equation}\label{E7} 
\left(\frac{p}{q}\right)_4 \left(\frac{q}{p}\right)_4 
  = \left(\frac{ac-bd}{q}\right).
\end{equation}

A rational reciprocity law equivalent to Burde's has already been
found by T. Gosset \cite{5}, who showed that, for primes $p$ and $q$
as above,
\begin{equation}\label{E8} 
\left(\frac{q}{p}\right)_4 \equiv 
  \left(\frac{a/b-c/d}{a/b+c/d}\right)^{(q-1)/4} \bmod q. 
\end{equation}
Multiplying the numerator and denominator of the term on the
right side of (8) by $a/b+c/d$ and observing that $c^2/d^2 \equiv 
-1 \bmod q$ yields
\begin{align*}
 \Big(\frac{q}{p}\Big)_4 & \equiv 
 \Big(\frac{a^2/b^2+1}{q}\Big)_4\Big(\frac{a/b+c/d}{q}\Big)
 = \Big(\frac{p}{q}\Big)_4 \Big(\frac{b}{q}\Big)
    \Big(\frac{a/b+c/d}{q}\Big) \\
 & = \Big(\frac{p}{q}\Big)_4 \Big(\frac{a+bc/d}{q}\Big)
   = \Big(\frac{p}{q}\Big)_4 \Big(\frac{d}{q}\Big)
   \Big(\frac{ad+bc}{q}\Big) \bmod q, 
\end{align*}
which is Burde's reciprocity law since $(\frac{2d}{q}) = +1$.

A more explicit form of Burde's reciprocity law for composite values
of $p$ and $q$ has been given by L. R\'edei \cite{8}; letting
$n = pq = A^2+B^2$ in \cite[(17)]{8}, we find
$A = ac-bd, \, B = ad+bc,$ and his reciprocity formula 
\cite[(23)]{8} gives (7).

Yet another version of Burde's law is due to A. Fr\"ohlich \cite{4};
he showed
\begin{equation}\label{E9}
 \left(\frac{p}{q}\right)_4 \left(\frac{q}{p}\right)_4 
  =\left(\frac{a+bj}{q}\right) = \left(\frac{c+di}{p}\right),
\end{equation}
where $i$ and $j$ denote rational numbers such that $i^2 \equiv -1
\bmod p$ and $j^2 \equiv -1 \bmod q$. Letting $i=a/b$ and $j=c/d$
and observing that $(\frac{a}{p}) = (\frac{c}{q}) = +1$,
we find that (9) is equivalent to (7).

The reciprocity theorem of Lehmer \cite{6,7} is even older; it can be 
found in Dirichlet's paper \cite{2} as Th\'eor\`eme I and II;
Dirichlet's ideas are reproduced in the charming book of Venkow  
\cite{10} and may be used to give proofs for other rational 
reciprocity laws using nothing beyond quadratic reciprocity.

\medskip 
\noindent {\bf Remark.} 
The author has recently generalized Scholz's 
reciprocity law to all number fields with odd class number in the
strict sense.

\end{document}